\documentclass[12pt]{article}
\usepackage{amssymb,amsmath,amsfonts,amsthm,amsxtra,setspace,enumerate}
\usepackage[T1]{fontenc}
\usepackage{lmodern}
\usepackage{comment}
\usepackage{enumitem}

\usepackage[pagewise,displaymath,mathlines]{lineno}

\usepackage{titlesec}
\titlelabel{\thetitle.\ }

\newcommand{\noopsort}[1]{}
\newcommand{\st}[2]{\ensuremath{\{#1\, :\, #2\}}}
\newcommand{\seq}[3]{\ensuremath{\st{#1_{#2}}{#3}}}
\newcommand{\insect}{\cap}

\newcommand{\union}{\cup}

\newtheorem{prob}{Problem}

\newtheorem{theorem}{Theorem}[section]
\newtheorem{cor}[theorem]{Corollary}

\newsavebox{\Prfref}

\newsavebox{\prfref}

\newtheoremstyle{ref}
{\topsep}	
{\topsep}	
{\it}
{}
{}
{}
{ }
{\thmname{{\bfseries#1}}\thmnumber{ \textbf{#2\thmnote{\rm #3}\textbf .}}}

\theoremstyle{ref}
\newtheorem{lem}[theorem]{Lemma}
\newtheorem{thm}[theorem]{Theorem}

\newtheorem{defn}[theorem]{Definition}

\newtheoremstyle{nnref}
{\topsep}	
{\topsep}	
{}
{}
{}
{}
{ }
{\thmname{\textbf{#1}\thmnote{\textrm{ #3}}\textbf{.}}}

\theoremstyle{nnref}

\theoremstyle{definition}
\newtheorem*{rmk}{Remark}

\begin{document}

\title{The Strength of Menger's Conjecture}
\author{Franklin D. Tall{$^1$}, Stevo Todorcevic{$^2$}, Se{\c{c}}il Tokg\"oz{$^3$}}

\footnotetext[1]{Research supported by NSERC grant A-7354.\vspace*{2pt}}
\footnotetext[2]{Research supported by grants from CNRS and NSERC 455916.}
\footnotetext[3]{Research supported by T\"UB\.{I}TAK  grant 2219.\vspace*{2pt}}
\date{\today}
\maketitle

\begin{abstract}
\noindent Menger conjectured that subsets of $\mathbb R$ with the Menger property must be $\sigma$-compact. While this is false when there is no restriction on the subsets of $\mathbb R$,  for projective subsets it is known to follow from the Axiom of  Projective Determinacy, which has considerable large cardinal consistency strength. We note that 
in fact, Menger's conjecture for projective sets has consistency strength of only an inaccessible cardinal.

\end{abstract}

\renewcommand{\thefootnote}{}
\footnote
{\parbox[1.8em]{\linewidth}{  2020 MSC. Primary 03E15, 03E35, 03E60, 54A35, 54D20, 54H05; Secondary 03E45.}}\\
\renewcommand{\thefootnote}{}
\footnote
{\parbox[1.8em]{\linewidth}{Keywords and phrases: Menger, Hurewicz, $\sigma$-compact, co-analytic, projective set of reals, $L(\mathbb{R})$, Hurewicz Dichotomy.}}

\section{Introduction}
In $1924$, Menger \cite{Menger} introduced a topological property for metric spaces which he referred to as \textbf{ ``property E''}.  Hurewicz \cite{Hur27} reformulated property E as the following, nowadays called the \textbf{Menger} property:

\begin{defn} A space $X$ is \emph{Menger} if whenever $\{{\mathcal U}_{n}\}_{n\in\omega}$ is a sequence of open covers, there exist finite ${{\mathcal V}}_{n}\subseteq{\mathcal U}_{n}$, $n\in\omega$, such that $\bigcup_{n\in\omega} {\mathcal V}_{n}$ is a cover of $X$.
\end{defn}

The \textbf{Hurewicz} property is intermediate between Menger and $\sigma$-compact.

\begin{defn} 
A space $X$ is \emph{Hurewicz} if for any sequence $\{{\mathcal U}_n\}_{n \in\omega}$ of open covers of $X$ there are finite
sets $\mathcal V_n\subseteq \mathcal U_n$ such  that $\{\bigcup\mathcal V_n:n\in\omega\}$ is a  $\gamma$-cover of $X$, where an infinite open cover $\mathcal{U}$ is a {\it$\gamma$-cover} if for each $x\in X$ the set $\{U\in {\mathcal{U}}: x\not\in U\}$ is finite.
\end{defn}

\textup{An equivalent definition (for completely regular spaces) is that} 
\emph{a space $X$ is Hurewicz if and only if for each \v{C}ech-complete space $Z \supseteq X$, 
there is a $\sigma$-compact space $Y$ such that $X \subseteq Y \subseteq Z$} \cite{Tall2011}, \cite{BZ}.

There has recently been interest in the question of whether ``definable'' Menger spaces --- and, more specifically, Menger sets of reals ---  are $\sigma$-compact. See e.g., \cite{Tallnew,TTo,Tok}. Hurewicz \cite{Hur25} refuted  under the Continuum  Hypothesis Menger's conjecture \cite{Menger} that  Menger subsets of $\mathbb R$ are $\sigma$-compact. 
Just et al.~\cite{COC2} refuted Hurewicz's conjecture that Hurewicz sets of reals are $\sigma$-compact, and hence also refuted Menger's conjecture in ZFC. A ZFC counterexample to Menger's conjecture was earlier produced by Chaber and Pol \cite{CP} in an unpublished note.  More natural examples were produced by Bartoszy\' nski and Shelah \cite{BaSh01}, and later Tsaban and Zdomskyy \cite{TsZ}.  A convenient source for examples differentiating these three properties is the survey paper \cite{Ts}. 

Hurewicz \cite{Hur25} proved that analytic Menger subsets of  $\mathbb R$ are $\sigma$-compact; this was later extended to arbitrary  Menger analytic spaces by Arhangel'ski{\u\i} \cite{A}. Hurewicz  \cite{Hur25} also proved this for completely metrizable spaces; this was extended to \v{C}ech-complete spaces in \cite{TTo}. That determinacy hypotheses suffice to imply more complicated ``definable'' Menger sets of reals (e.g.~co-analytic ones) are $\sigma$-compact was first noticed in \cite{MF} and stated explicitly in \cite{Tall2011}. See also \cite{TTs} and \cite{BT}.

Determinacy hypotheses have considerable large cardinal strength, so it is of interest to compute the exact consistency strength of such propositions as e.g.~``every co-analytic (projective) Menger set of reals is $\sigma$-compact''.  

We shall consider three primary families of ``definable'' sets of reals: the co-analytic sets, the projective sets, and those sets of reals which are members of $L(\mathbb{R})$. The co-analytic sets are just the complements of analytic sets; the projective sets are obtained by closing the Borel sets under complementation and continuous real-valued image. They are arranged in a hierarchy -- see Kechris \cite{Kech} for notation and properties.  The co-analytic sets are also called the $\Pi^1_1$-sets.  
$L(\mathbb{R})$ is the \textit{constructible closure} of $\mathbb{R}$. It is the smallest inner model of ZF with $\mathbb{R}$ as a member.  See e.g.~Kanamori \cite{Kan94} or Moschovakis \cite{Mosch} for its properties.  It is frequently studied in its own right; for us, $\mathcal{P}(\mathbb{R}) \cap L(\mathbb{R})$ -- those sets of reals that are in $L(\mathbb{R})$ -- is a convenient large family of definable sets of reals that includes the projective sets and much more. For readers unfamiliar with definability, we point out that the process of constructing a Borel set, projective set, etc.~can be encoded as a sequence of (sequences of \dots) operations on rational intervals, and hence as a real number. 
Our main result is:

\begin{thm}\label{bil}
The following are equiconsistent: 
\begin{enumerate}[label=\alph*)]
\item there is an inaccessible cardinal, 
\item every Menger co-analytic set of reals is $\sigma$-compact, 
\item every Menger projective set of reals is $\sigma$-compact, 
\item every Menger set of reals in $L(\mathbb{R})$ is $\sigma$-compact.
\end{enumerate}
\end{thm}

As is common in descriptive set theory, we will use $\mathbb{R}$ or the Cantor set as convenient, since e.g. there is a  co-analytic Menger non-$\sigma$-compact subset of  
$\mathbb{R}$ if and only if  there is one included in the Cantor set.

\section{The Hurewicz Dichotomy}
 
  A classical phenomenon,  the \emph{Hurewicz Dichotomy}, was first investigated by Hurewicz \cite{Hur28} and later extended by Kechris, Louveau and Woodin \cite{KLW}. See e.g.~Section $21.F$ of  \cite{Kech}. Here is one version of the Hurewicz Dichotomy.
\\

\textbf{Hurewicz Dichotomy (HD).}  Let  $X$ be a Polish (separable completely metrizable) space and $A\subseteq X$ an  analytic set.
If $A$ is not $\sigma$-compact, then there is a Cantor set $K\subseteq  X$ such that $K\cap A$ is dense in $K$ and homeomorphic to  $\mathbb{P}$, the space of irrationals, and $K\setminus A$ is countable dense in $K$ and homeomorphic to $\mathbb{Q}$, the space of rationals.  

\begin{defn} Let $\Gamma$ be a subset of the power set of $\mathbb{R}$. $\mathrm{\mathbf{HD(\Gamma)}}$ is the assertion obtained from $\mathbf{HD}$ by replacing ``analytic''  by `` a member of $\Gamma$'' .
\end{defn}

\begin{thm}[~\cite{Hur28}] If  $\Gamma$ is a collection of subsets of $\mathbb{R}$ satisfying $\mathrm{\mathbf{HD(\Gamma)}}$ as above, then every Menger member of $\Gamma$ is $\sigma$-compact.

\end{thm}
\begin{proof} Let $A$ be a member of $\Gamma$. Suppose $A$ is not  $\sigma$-compact.  By $\mathbf{HD(\Gamma)}$,  there is a Cantor set $K$ such that $K\subseteq \mathbb{R}$  and  $K\cap A$ is homeomorphic to $\mathbb{P}$ . But $K\cap A$ is a closed subset of $A$ and  $\mathbb{P}$ is not Menger \cite{Hur27}; since that property is closed-hereditary,  $A$ cannot be Menger.
\end{proof}


\begin{rmk}
The proof that an inaccessible suffices to prove the consistency with ZFC of $\mathbf{HD(L(\mathbb{R}))}$ (of course we mean $\mathcal{P}(\mathbb{R}) \cap L(\mathbb{R})$) and hence that Menger projective sets are $\sigma$-compact can actually be found in Di Prisco--Todorcevic \cite{DiPT}.
This may not be obvious to the casual reader, since the authors of \cite{DiPT} are interested in $L(\mathbb{R})$ and other models not satisfying the Axiom of Choice. However the results about $L(\mathbb{R})$ satisfying various principles such as the Hurewicz Dichotomy for all sets of reals can be interpreted as ZFC results about sets of reals that happen to be in $L(\mathbb{R})$.\\
\end{rmk}

\begin{rmk}
Solovay \cite{Sol} was the first to realize the usefulness of the model $L(\mathbb{R})$ as computed in the forcing extension  obtained by collapsing an inaccessible cardinal to $\omega_1$ via finite conditions to problems of descriptive set theory such as for example the problem of Lebesgue measurability of projective sets of reals. Feng's paper \cite{Fe} contains various interesting results about this {\bf Solovay model}  $L(\mathbb{R})$ , e.g. ~ extensions of the fact that uncountable sets of reals which are in $L(\mathbb{R})$ must include a perfect set. Solovay models are further explored in Di Prisco--Todorcevic \cite{DiPT} and Todorcevic \cite{T}.  Section 4 plus point \#6 on p.~249 of \cite{DiPT} prove that $\mathbf{HD(L(\mathbb{R}))}$ holds in such models.

\end{rmk}


\section{The inaccessible is necessary}

To prove Theorem \ref{bil}, by the above Remarks  it  more than suffices to show:

\begin{thm}\label{mengerinacc}
If  $\omega_{1}^{L[a]}=\omega_1$ for some $a \in \mathbb{R}$, then there is a co-analytic set of reals which is  Hurewicz but not $\sigma$-compact.
\end{thm}

 The reason is that then \ref{bil}b) (and hence  \ref{bil}c)   and  \ref{bil}d))  imply $ \omega_{1}^{L[a]}\textless \omega_1$ for all $a \in \mathbb{R}$ and so $ \omega_{1}^{L}\textless \omega_1$. But then   $\omega_1$ is inaccessible in $L$, so it's consistent there is an inaccessible. We shall rely on the following version of a standard fact (see \cite{Kan94}, p.~171).

\begin{lem}\label{key} Assume $\omega_{1}^{L[a]}=\omega_1$ for some $a \in \mathbb{R}$. Then $\omega^{ \omega}\cap L[a]$ ordered by the relation $\leq^*$
of eventual dominance  has a co-analytic  $\omega_1$-scale, i.e., a  cofinal subset $A$ which is well-ordered by $ \leq^*$ in order type $\omega_1.$
\end{lem}

\begin{proof}[Proof of Theorem \ref{mengerinacc}]
Let $A$ be the co-analytic set given by Lemma \ref{key}. 
We know that  $A$ is not $\sigma$-compact and in fact not Borel. 
This follows from the standard fact that a Borel well-founded relation on a Borel set of reals has countable rank (see {~\cite[p.~239]{Kech}}). 
If $A$ is Hurewicz, then the proof of Theorem \ref{mengerinacc} is finished. Otherwise, by a theorem of Hurewicz \cite{Hur25}, there is a continuous
mapping $f: A\rightarrow \omega^\omega$ whose range is unbounded in $(\omega^\omega, \leq^*).$ 
The map $f$ extends to a continuous map on a $G_\delta$-superset of $A$. So   there is a Borel map (also called a \emph{measurable map})
$g:\omega^\omega\rightarrow \omega^\omega$ such that $g\upharpoonright A=f$. See Theorem $12.2$ in  \cite{Kech}  for more details.  Let $b\in \omega^\omega$ code both $a$ and the map $g.$ Then $\omega^{ \omega}\cap\, L[b]$ is unbounded in   $(\omega^\omega, \leq^*)$.  Applying  Lemma \ref{key} again, we obtain a co-analytic $\omega_1$-scale $B$ in $(\omega^\omega\cap L[b],\leq^*).$ Since $\omega^\omega\cap\, L[b]$ is unbounded in   $(\omega^\omega, \leq^*)$, that co-analytic $\omega_1$-scale $B$  is then a $\mathfrak{b}$-scale in $\omega^\omega$, i.e. an unbounded set $\{b_\alpha:\alpha<\mathfrak{b}\}$ such that the enumeration is increasing with respect to $\leq^*$. By  \cite[Theorem 3.3]{Ts}, $B\cup\omega^{<\infty}$ is Hurewicz, so the proof  of Theorem \ref{mengerinacc} is finished and hence so is the proof of Theorem \ref{bil}.
\end{proof}

In \cite{TZ}, Tall and Zdomskyy show questions about whether Menger definable sets of reals are $\sigma$-compact are essentially equivalent to questions about whether \textit{completely Baire}  definable sets of reals are Polish, where a space is completely Baire if each closed subspace satisfies the Baire Category Theorem.  A result needed for their work (which was written after seeing an early preprint of this paper) is that: 

\begin{thm}\label{uc}
If it is consistent there is an inaccessible cardinal, it is consistent that every completely Baire projective set of reals is Polish.
\end{thm}

\begin{proof}
We have stated in Section 2  that the consistency of an inaccessible yields the consistency of \textbf{HD(projective)}. From that it is easy to prove the desired result---see the proof of 21.21 from 21.18 in \cite{Kech}.
\end{proof}

\section{ A co-analytic gap theorem and the Hurewicz Dichotomy}

We shall need the following which is a  variant of Theorem 3 of \cite{T2}  given in \cite{T4}. For a family of subsets of $\omega$ denoted by $B$,  a subset $\Sigma\subseteq [\omega]^{< \omega}$ is called a \textbf{B-tree}  \cite{T2} if 
 
 \noindent(i) $\emptyset\in \Sigma$,\,
 (ii) for every $\sigma\in \Sigma$, the set $\{ i\in \omega: \sigma\cup \{i\}\in\Sigma\}$ is infinite and included in an element of $B$.\\

\begin{thm}\label{CAG}
\textbf{A co-analytic gap theorem} $\mathbf{(CAG)}$. Suppose $ \omega_{1}^{L[x]}\textless \omega_1$ for all $x \subseteq\omega$. Let $A$ and $B$ be two orthogonal families of subsets of $\omega$ closed downwards such that $A$ is co-analytic and $B$ is analytic or co-analytic. Then either $A$ is countably generated in   $B^\perp$ or there is a $B$-tree all  of whose branches are in $A$.
\end{thm}

\begin{proof}
For the convenience of the reader we sketch the argument from \cite{T4}. The proof is, in fact, a   straightforward  variation  of the proof of Theorem 3 of \cite{T2}. We start the proof by fixing a real $x$ and two downwards closed subtrees $T$ and $S$ of  $[\omega]^{<\omega}\otimes[\omega_1]^{<\omega}$ ordered by end-extension  and belonging to $L[x]$ for some $x\subseteq \omega$ such that $A=p[T]$ and $B=p[S]$. (See e.g. \cite[p.86]{M}.) For a subtree $U$ of $T$ and $t=(t_0, t_1)\in U,$ let $U(t)$ denote the subtree of $U$ consisting of all nodes of $U$ comparable to $t.$ For a downwards closed subtree $U$ of $T,$ let
$$\partial U=\{t\in U: \bigcup p[U(t)]\not\in B^\perp\}.$$
Note that by absoluteness, if $U$ belongs to $L[x]$ so does $\partial U.$

Let $T^{(0)}=T$, $T^{(\alpha+1)}=\partial T^{\alpha}$ and $T^{\lambda}=\bigcap_{\alpha<\lambda} T^{\alpha}$ for limit ordinal $\lambda.$ Let $\beta$ be the minimal ordinal $\alpha$ with the property that  $T^{(\alpha+1)}=T^{\alpha}.$ If $T^{\beta}\neq \emptyset$, then working as in the proof of Theorem 3 of \cite{T2}, we get a $B$-tree all of whose branches are in $A.$  If $T^{\beta}= \emptyset$, then  for every $a\in A$ there are $\alpha<\beta$ and $t\in T^{(\alpha)}\setminus T^{\alpha+1}$ such that $a\in p[T^{(\alpha)}(t)]$ and therefore
$a\subseteq b(\alpha, t)= \bigcup p[T^{(\alpha)}(t)].$ We have already noted that the trees of the form
$T^{(\alpha)}(t)$ belong to $L[x],$ so we have that the sets $b(\alpha, t)$ are also elements of $L[x].$
Since all these sets are in $B^\perp,$ we have that $A$ is generated by $B^\perp\cap L[x].$ Since by our assumption $ \omega_{1}^{L[x]}\textless \omega_1$,  this set is countable, so we conclude that $A$ is countably generated in $B^\perp.$
\end{proof}

\begin{rmk}
In \cite{T} it is shown that in the Solovay model the  conclusion of $4.1$   holds whenever $A$ and $B$ are definable from finite sets of reals and ordinals. Thus, in particular. the consistency of 1.3a) implies the consistency of 1.3d).

\end{rmk}

\begin{thm}\label{thm:OGA-wHD}
$\mathbf{CAG}$ implies $\mathbf{HD(\mathbf{\Pi^1_1})}$.
\end{thm}





%


\begin{proof}[Proof of Theorem \ref{thm:OGA-wHD}]
 Suppose $A$ is a co-analytic  not $\sigma$-compact set. We shall find a counterexample to the conclusion of the co-analytic gap theorem. First of all, we may assume $A$ is a subset of the Cantor set $2^\omega.$
 Let $\hat{A}$ be the collection of all infinite chains of the Cantor tree $2^{<\omega}$ whose union belongs to $A.$ 
 Note that $\hat{A}$ is a co-analytic collection of infinite subsets of $2^{<\omega}.$ Let $B=\hat{A}^\perp.$  We shall show that both alternatives of CAG fail for the gap $(\hat{A}, B).$  First of all note that $\hat{A}$ is not a countably generated ideal since otherwise $A$ would be a $\sigma$-compact set. Since $B^\perp=\hat{A}$ this shows that the first alternative of CAG fails for the gap $(\hat{A}, B).$  So, we are left with the alternative that there is a $B$-tree $\Sigma$ all of whose infinite branches are in $\hat{A}.$  Thus, $\Sigma$ is a collection of finite subsets of $2^{<\omega}$ such that $\emptyset\in \Sigma$ and if $t\in \Sigma,$ then $\Sigma(t)=\{\sigma\in 2^{<\omega}: t\cup\{\sigma\}\}$ is an infinite set belonging to $B.$ To view $\Sigma$ as a tree we order $2^{<\omega}$ in order type $\omega$ extending the partial ordering of end-extension in some natural way; we assume that for every $t\in \Sigma$, every $\sigma\in \Sigma(t)$ is above every $\tau\in t.$  This allows us to define a tree ordering on $\Sigma$ by letting $s\sqsubseteq t$ if and only if $s\subseteq t$ and every element of $s$ is smaller than every element of $t\setminus s$ in the $\omega$-ordering of $2^{<\omega}$ just fixed. Define a one-to-one mapping
 $\phi: 2^{<\omega}\rightarrow 2^{<\omega} $ as follows. The  definition is by  recursion on the $\omega$-ordering. Let $\phi(\emptyset)=\emptyset$. Suppose $\phi(\tau)$ is defined. Let $\tau[1]$ be the largest initial segment of $\tau$ with last digit $1$; if such initial segment does not exist, put $\tau[1]=\emptyset.$ Let $\phi(\tau^\frown 0)$ be the minimal available element of  $\Sigma(\phi(\tau))$ and let $\phi(\tau^\frown 1 )$ be the minimal available element of $\Sigma(\tau).$
 The following properties of  $\phi: 2^{<\omega}\rightarrow 2^{<\omega} $ are easy to verify.  If $x\in 2^\omega$ has infinitely many $1$'s, then the $\phi$-image of the infinite chain $c_x=\{ x\upharpoonright n: x(n)=1 \}$ is an infinite chain of $\Sigma$ and therefore an element of $\hat{A}.$ On the other hand, if $x\in 2^\omega$ is eventually $0$ then the $\phi$-image of the chain 
$c_x=\{ x\upharpoonright n: x(n)=1 \}\cup\{x\upharpoonright n: n>n_x\},$ where $n_x$ is the maximal integer where $x$ has a digit $1$, belongs to the family $B.$  Note that $x\mapsto \phi[c_x]$ is a continuous map from $2^\omega$ into the power-set of $2^{<\omega}$, viewed as $2^{2^{<\omega}}.$
 So if we let $P=\{\phi[c_x]: x\in 2^\omega\},$
 we get a copy of the Cantor set inside $2^{2^{<\omega}}$ such that 
 $P\setminus \hat{A}$ is a countable dense subset of $P$. It follows that the complement of $\hat{A}$ in   $2^{2^{<\omega}}$ includes a closed copy of the rationals, so  it can't be analytic, a contradiction. This finishes the proof. 
 \end{proof}

\medskip

We have the following consequence.

\begin{cor}\label{sonu}
The following are equivalent:
\begin{enumerate}[label=\alph*)]
	\item $ \omega_{1}^{L[a]}\textless \omega_1$ for all $a\in\mathbb{R}$,
	\item co-analytic Menger subsets of $\mathbb{R}$ are $\sigma$-compact,
	\item co-analytic Hurewicz subsets of $\mathbb{R}$ are $\sigma$-compact.
\end{enumerate}
\end{cor}

 \begin{rmk} An early version of this paper (The Open Graph Axiom and Menger's Conjecture, Arxiv.org) claimed that the Open Graph Axiom (formerly known as the Open Coloring Axiom of \cite{T1}, but renamed to avoid confusion with the identically named axiom of \cite{ARS}) for co-analytic sets implied co-analytic (projective) Menger sets of reals are $\sigma$-compact. The proof was flawed, but the result is true because of Theorems \ref{CAG} and \ref{thm:OGA-wHD} and the following result.

\end{rmk}



 \begin{thm} [~ \cite{Fe}] The following are equivalent:
 \item 1. $\omega_{1}^ { L[a] }<\omega_1$  for all $a \in \mathbb{R}$, 
 \item 2. $\omega_1$ is inaccessible in $L[a]$, any $a \in \mathbb{R}$,
 \item 3. $\mathbf{OGA^*( \Pi^1_1)}$\, (the co-analytic axiom referred to above). 
 
 \end{thm}
 
  \cite{TZ} was in press as we were revising  this paper; the reference to    $\mathbf{OGA^*}$  there is not incorrect but is irrelevant.

\begin{rmk}
 There is already a rich body results about the inner model $L(\mathbb{R})$ especially when it is a Solovay model (i.e., computed in the forcing extensions of the Levy collapse of an inaccessible cardinal to $\omega_1$) and therefore fails to satisfy
 the Axiom of Choice but rather satisfies strong descriptive set-theoretic regularity properties for all sets of reals.  In \cite{DiPT} , Di Prisco and Todorcevic compare the Solovay model  $L(\mathbb{R})$ and its (forcing) extension $L(\mathbb{R})[U]$, which is obtained from that $L(\mathbb{R})$ by adjoining a selective ultrafilter $U$. They note that $\mathbf{OGA^*(\mathbf{\mathcal{P}(\mathbb{R}}))}$ holds in the latter model, but that $\mathbf{HD(\mathcal{P}(\mathbb{R}))}$ does not. Thus one has to be careful about asserting implications from various forms of $\mathbf{OGA^*}$ to corresponding forms of $\mathbf{HD}$. The counterexample  to   $\mathbf{HD}$ is defined from the generic ultrafilter $U$ and, therefore, the model $L(\mathbb{R})[U]$  satisfies  $\mathbf{OGA^*(L(\mathbb{R})[U])}$ but fails to satisfy $\mathbf{HD(L(\mathbb{R})[U])}$.

\end{rmk}

Here is another version of Theorem \ref{uc} :


\begin{thm}\label{thm:analyticGdelta}
If $\omega_1^{L[a]} < \omega_1$ for all $a \in \mathbb{R}$, then every analytic completely Baire subset of $\mathbb{R}$ is a $G_\delta$.
\end{thm}

Theorem \ref{thm:analyticGdelta} is used in \cite{TZ}, replacing    ``a $G_\delta$''   by  ``Polish'', which is obviously equivalent.

\begin{lem}[{ \cite{Hur28},\cite{Mi}}] $B \subseteq \mathbb{R}$ is completely Baire if and only if $B$ does not include a closed copy of $\mathbb{Q}$.
\end{lem}

 \begin{proof}[ Proof of Theorem \ref{thm:analyticGdelta}] 
Suppose $B$ is completely Baire, analytic, but not a $G_\delta$. Then by $\mathbf{HD(\mathbf{\Pi^1_1})}$  , since $\mathbb{R}\setminus B$ is co-analytic and not $\sigma$-compact, there is a compact $K$ with $K\cap B$ homeomorphic to $\mathbb{Q}$. But $K\cap B$ is closed in $B$, contradicting $B$ being completely Baire. \end{proof}

We can now add an additional clause to Corollary \ref{sonu}:\\

e) every analytic completely Baire subset of $\mathbb R$ is a $G_\delta$.\\

The reason is that in \cite{TZ} it is established that if analytic completely  Baire subsets of $\mathbb R$ are Polish, then  co-analytic Menger subsets of  $\mathbb R$ are $\sigma$-compact.

\begin{prob} Is there a model in which every co-analytic Menger set of reals is $\sigma$-compact, but there is a projective Menger set of reals which is not $\sigma$-compact?
\end{prob}

\begin{prob}
	If there is a  co-analytic Menger subset of $\mathbb{R}$ which is not $\sigma$-compact, is there one which is not Hurewicz?
\end{prob}


  In conclusion, we thank the referee for pointing out several inaccuracies in the previous version of this note.

{\rm Franklin D. Tall, Department of Mathematics, University of Toronto, \\Toronto, Ontario M5S 2E4, CANADA}\\
{\it e-mail address:} {\rm tall@math.utoronto.ca}\\

{\rm Stevo Todorcevic, Department of Mathematics, University of Toronto, \\Toronto, Ontario M5S 2E4, CANADA}\\
{\it e-mail address:} {\rm stevo@math.utoronto.ca}\\

{\rm Se{\c{c}}il Tokg\"oz, Department of Mathematics, Hacettepe University,\\ Beytepe, 06800, Ankara, TURKEY}\\
{\it e-mail address:} {\rm secilc@gmail.com}

\end{document}